\documentclass[oneside, reqno, 11pt]{amsart}
\usepackage{amssymb,mathrsfs,amstext, enumerate, comment}
\usepackage[noadjust]{cite}
\usepackage[plainpages=false,colorlinks,hyperindex,pdfpagemode=None,bookmarksopen,linkcolor=blue,citecolor=blue,urlcolor=blue]{hyperref}

\newtheorem{theorem}{Theorem}[section]

\newtheorem{corollary}[theorem]{Corollary}
\newtheorem{lemma}[theorem]{Lemma}

\theoremstyle{definition}

\newcommand{\TT}{{\mathrm{t}}}

\theoremstyle{remark}

\newtheorem{problem}[theorem]{Problem}

\numberwithin{equation}{section}

\def\Span{{\rm span}\,}

\def\bV{{\bf V}}

\def\Span{{\rm span}\,}

\def\diag{{\rm diag}\,}

\def\dim{{\rm dim}\,}

\def\IC{{\mathbb C}}
\def\IF{{\mathbb F}}
\def\IR{{\mathbb R}}

\def\bM{{\mathbf M}}

\newcommand{\be}{\mathbf{e}}
\newcommand{\bt}{\mathbf{t}}
\newcommand{\bu}{\mathbf{u}}
\newcommand{\bv}{\mathbf{v}}
\newcommand{\bx}{\mathbf{x}}
\newcommand{\by}{\mathbf{y}}
\newcommand{\bz}{\mathbf{z}}
\newcommand{\fU}{\mathfrak{U}}

\newcommand{\Pk}{\operatorname{Pk}}

\begin{document}
\openup .9\jot
\title{Linear maps preserving $\ell_p$-norm parallel vectors}

\dedicatory{In the memory of  Professor Peter Michael Rosenthal (June 1, 1941–-May 25, 2024) }

\author[Li, Tsai, Wang  and Wong]
{Chi-Kwong Li, Ming-Cheng Tsai, Ya-Shu Wang  \and Ngai-Ching Wong}

\address[Li]{Department of Mathematics, The College of William
\& Mary, Williamsburg, VA 13187, USA.}
\email{ckli@math.wm.edu}

\address[Tsai]{General Education Center, National Taipei University of  Technology, Taipei 10608, Taiwan.}
\email{mctsai2@mail.ntut.edu.tw}

\address[Wang]{Department of Applied Mathematics, National Chung Hsing University, Taichung 40227, Taiwan.}
\email{yashu@nchu.edu.tw}

\address[Wong]{Department of Applied Mathematics, National Sun Yat-sen University, Kaohsiung, 80424, Taiwan; Department of Healthcare
Administration and Medical Information,
Kaohsiung  Medical University, 80708 Kaohsiung, Taiwan.}
  \email{wong@math.nsysu.edu.tw}
\date{}

\begin{abstract}
Two vectors $\bx, \by$ in a normed vector space  are parallel   if
there is a scalar $\mu$ with $|\mu| = 1$ such that
$\|\bx+\mu \by\| = \|\bx\| + \|\by\|$; they form a triangle equality attaining (TEA) pair  if
$\|\bx+\by\| = \|\bx\| + \|\by\|$. In this paper, we characterize linear maps on $\IF^n=\IR^n$ or $\IC^n$, equipped with the $\ell_p$-norm
 for $p \in [1, \infty]$, preserving parallel pairs or preserving TEA pairs.
Indeed,  any linear map will preserve parallel pairs
and TEA pairs when $1< p <\infty$.
For the $\ell_1$-norm, TEA  preservers form a semigroup of matrices in which each row has at most one
nonzero entries; adding rank one matrices to this semigroup will be the semigroup of parallel preserves.
For the $\ell_\infty$-norm,  a nonzero TEA preserver, or a parallel preserver of rank greater than one, is always a
multiple of an $\ell_\infty$-norm isometry,
except when $\IF^n = \IR^2$.  We also have a  characterization for the exceptional case.
The results are extended to    linear maps of
the  infinite dimensional spaces $\ell_1(\Lambda)$, $c_0(\Lambda)$ and $\ell_\infty(\Lambda)$.
\end{abstract}


\maketitle


\section{Introduction}\label{s:Intro}

In a normed vector space $(\bV,\|\cdot\|)$ over the real field $\IF=\IR$ or the complex field $\IF=\IC$,
two vectors $\bx, \by \in \bV$ form a  \emph{parallel} pair if
\begin{equation}\label{parallel}
\|\bx+\mu \by\| = \|\bx\| + \|\by\| \quad \hbox{ for some } \mu \in \IF \hbox{ with } |\mu| = 1;
\end{equation}
the vectors $\bx,\by \in \bV$ form a \emph{triangle equality attaining} (\emph{TEA}) pairs if
\begin{equation}\label{TEA}
\|\bx+\by\| = \|\bx\| + \|\by\|.
\end{equation}
We are interested in those linear maps $T: (\IF^n, \|\cdot\|) \rightarrow (\IF^n,\|\cdot\|)$  preserving
parallel pairs, namely,
\begin{equation}\label{g-pp-pre}
(T\bx,T\by) \hbox{ is a parallel pair}\quad  \hbox{whenever}\quad  (\bx,\by)  \hbox{ is a parallel pair,}
\end{equation}
or preserving TEA pairs, namely,
\begin{equation}\label{g-te-pre}
(T\bx,T\by) \hbox{ is a TEA pair}\quad  \hbox{whenever}\quad (\bx,\by) \hbox{ is a TEA pair.}
\end{equation}

Recall that a norm $\|\cdot\|$ on $\bV$ is {strictly convex} if for any two nonzero  vectors
$\bx, \by \in \bV$ satisfying
$\|\bx+\by\| = \|\bx\| + \|\by\|$ we have $\bx = t\by$ for some $t > 0$.
It is easy to see that if the norm $\|\cdot\|$ is strictly convex, then any linear map $T$
will satisfy (\ref{g-pp-pre}) and (\ref{g-te-pre}). For example, the $\ell_p$-norms on $\IF^n$
defined by
$\|\bx\|_p = (\sum_{j=1}^n |x_j|^p)^{1/p}$
when $p \ge 1$, and
$\|\bx\|_\infty = \max\{ |x_j|: 1 \le j \le n\}$, for $\bx = (x_1, \dots, x_n)^{\TT}$,
are strict convex except for $p = 1, \infty$.

In this paper, we characterize linear maps on $\IF^n$
 preserving parallel pairs or preserving TEA pairs with respect to the
$\ell_1$-norm and $\ell_\infty$-norm, respectively.
In the following,  let $\{\be_1,\ldots, \be_n\}$ denote the standard basis for $\IF^n$, and $\bM_n$ be the algebra of
$n\times n$ matrices with entries from $\IF$.  We  identify  linear maps from  $\IF^n$ to  $\IF^n$ with matrices in $\bM_n$.
For $\bu\in \IF^n$ and $A \in \bM_n$, we let $\bu^\TT$ and $A^\TT$ denote their transposes.
Below are our findings.  The proofs are given in Section \ref{s:finite-dim}.

\begin{theorem}
\label{prop:1-preservers}\label{p-3.2}
Let $T: (\IF^n,  \|\cdot\|_1)\rightarrow (\IF^n,  \|\cdot\|_1)$ be a linear map.
\begin{itemize}
\item[{\rm (a)}] $T$ preserves TEA pairs if and only if each row of $T$  has at most one nonzero entry.
\item[{\rm (b)}] $T$ preserves parallel pairs if and only if each row of $T$ has at most one nonzero entry,
or $T = {\bv}{\bu}^\TT$
for some column vectors  ${\bu}, {\bv} \in \IF^n$.
\end{itemize}
\end{theorem}

By the above result,
for the $\ell_1$-norm, TEA  preservers form a semigroup of matrices such that each row has at most one
nonzero entries; adding rank one matrices to this semigroup will be the semigroup of parallel preserves.

Recall that a matrix  in $\bM_n$ is a \emph{monomial} matrix
if each row and each column of it has exactly
one nonzero entry, and
  a monomial matrix  is a \emph{generalized permutation matrix} if
 all its nonzero  entries have
modulus one.
By   Theorem \ref{prop:1-preservers},
 an invertible linear map $T$ preserves TEA
or parallel pairs of $(\IF^n,  \|\cdot\|_1)$ if and only if it is a monomial
matrix.
For $(\IF^n,\|\cdot\|_\infty)$, we have the following.

\begin{theorem} \label{p-3.2old}
Let  $T: (\IF^n, \|\cdot\|_\infty)\rightarrow (\IF^n, \|\cdot\|_\infty)$ be a linear map.
 \begin{enumerate}
 \item [{\rm (a)}]
 $T$ preserves parallel pairs if and only if there is $\gamma \ge 0$ and a generalized permutation matrix $Q$
 such that one of the following forms holds:

 \begin{enumerate}
  \item[{\rm (a.1)}] $T$ has the form $\bx \mapsto \gamma Q\bx$,
  \item[{\rm (a.2)}] $n = 2$ and $T$ has the form $\bx \mapsto \gamma CQ\bx$, where
  $C =\begin{pmatrix} 1 & \beta \\ \overline{\beta} & 1\end{pmatrix}$
for some scalar $\beta$ with $|\beta|<1$.
  \item[{\rm (a.3)}] $T=\bv{\bu}^\TT$ for some nonzero column vectors ${\bu}, {\bv}\in \IF^n$.
\end{enumerate}
\item[{\rm (b)}]  $T$  preserves
TEA pairs if and only if
one of the following holds.

\begin{enumerate}
\item[{\rm (b.1)}] $T$ has the form in {\rm (a.1)}.

\item[{\rm (b.2)}] $\IF^n = \IR^2$ and $T$ has the form in {\rm (a.2)}.

\item[{\rm (b.3)}] $\IF^n = \IR^2$ and $T$ has the form in {\rm (a.3)}, where
$\bu=(u_1, u_2)^{\TT}$ with $|u_1| = |u_2|$.
\end{enumerate}
 \end{enumerate}
\end{theorem}


\begin{corollary} \label{p-3.1}
Let  $n \ge 3$.
The following conditions are
equivalent to each other for a nonzero linear map  $T:(\IF^n, \|\cdot\|_\infty) \rightarrow (\IF^n, \|\cdot\|_\infty)$.
\begin{itemize}
\item[{\rm (a)}]   $T$   preserves TEA pairs.
\item[{\rm (b)}]   $T$
preserves  parallel pairs and its range space has dimension larger than one.
\item[{\rm (c)}] There is $\gamma > 0$ and a generalized permutation
matrix $Q\in \bM_n$ such that
$T$ has the form $\bx\mapsto \gamma Q\bx$.
\end{itemize}
\end{corollary}

It is known that an isometry $T$ for the normed space $(\IF^n,\|\cdot\|_p)$, with $p \in [1,\infty]$ and $p \ne 2$,
is a {generalized permutation matrix}; see, e.g., \cite{Ket}.
By Corollary \ref{p-3.1},  for $(\bV, \|\cdot\|)
= (\IF^n,\|\cdot\|_\infty)$ with $n \geq 3$, a bijective linear map satisfies
(\ref{g-pp-pre}) if and only if it satisfies
(\ref{g-te-pre}); such a map is a scalar multiple of an $\ell_\infty$-isometry.
On the other hand, by the remark after Theorem \ref{prop:1-preservers}, for $(\IF^n, \|\cdot\|_1)$
with $n \ge 2$, a bijective linear map $T$ satisfying \eqref{g-pp-pre} (or, equivalently,
 \eqref{g-te-pre}) if and only if  $T$ is a monomial matrix.
Note, however, that by Theorem \ref{prop:1-preservers}, the range space of  linear TEA/parallel pair preservers on $(\IF^n, \|\cdot\|_1)$
  can have arbitrary dimension in contrast to the $\ell_\infty$--norm case.

More generally, it would be interesting to determine  a norm $\|\cdot\|$ on $\IF^n$ such that the set of bijective
linear maps $T: \IF^n \to \IF^n$ satisfying \eqref{g-pp-pre} or \eqref{g-te-pre}
consisting of scalar multiples of isometries.
For more results concerning the study of parallel pairs and their preservers; see
\cite{BB01,NT02,Seddik07,Wojcik,ZM16,Zamani} and \cite{Ket,LTWW}.

In Section \ref{s:inf-dim}, we provide similar characterizations of parallel/TEA pair   linear preservers of the infinite dimensional
Banach ``sequence spaces'' $\ell_1(\Lambda)$, $c_0(\Lambda)$ and $\ell_\infty(\Lambda)$, where the index set $\Lambda$ can be
uncountably infinite.  The case for the $\ell_1$--norm is similar, while for the $\ell_\infty$--norm case, we need to
assume the   linear map  $T$ of $c_0(\Lambda)$ or $\ell_\infty(\Lambda)$
preserves  parallel/TEA pairs  in both directions; namely,
$$
(T\bx,T\by) \hbox{ is a parallel/TEA pair}\quad  \hbox{if and only if}\quad  (\bx,\by)  \hbox{ is a parallel/TEA pair.}
$$
Under this stronger assumption, we see that $\ell_\infty$--norm parallel/TEA linear preservers
of $c_0(\Lambda)$ and $\ell_\infty(\Lambda)$   are   scalar multiples of isometries,
in line with the finite dimensional case.

\section{Preservers of parallel/TEA vectors}\label{s:finite-dim}

We begin with the following observation.

\begin{lemma}\label{lem:seqn-para}
Let $\bx=(x_1, \dots, x_n)^{\TT}, \by=(y_1, \dots, y_n)^{\TT} \in \IF^n$.
\begin{enumerate}
  \item[{\rm (a)}] $\bx,\by$ are parallel (resp.\ TEA)  with respect to the $\ell_1$--norm if and only if
  there is a unimodular scalar $\mu$ (resp.\ $\mu=1$) such that $\mu \overline{x_k}y_k\geq 0$ for all $k=1,\ldots, n$.
  \item[{\rm (b)}] $\bx,\by$ are parallel (resp.\ TEA) with respect to the $\ell_\infty$--norm if and only if $\|\bx\|_\infty = |x_k|$
  and $\|\by\|_\infty = |y_k|$ (resp.\ such that $\overline{x_k}y_k \geq 0$) for some $k$ between $1$ and $n$.
  \end{enumerate}
\end{lemma}

As direct consequences of Lemma \ref{lem:seqn-para}, a linear map $T$ of $\mathbb{F}^n$
 preserves $\ell_1$--norm (resp.\ $\ell_\infty$--norm) TEA pairs  if, and only if, $PTQ$ does for any monomial matrices
 (resp.\ generalized permutation matrices) $P, Q$.


\begin{proof}[Proof of Theorem \ref{prop:1-preservers}]
(a)  Suppose each row of the $n\times n$ matrix $T$ has at most one nonzero entries. Then there are monomial matrices
$P, Q$ such that $PTQ = \begin{bmatrix}
                          \mathbf{T}_1 &  \mathbf{T}_2 &\cdots & \mathbf{T}_n
                        \end{bmatrix}$
in which the column vectors  $\mathbf{T}_1,  \mathbf{T}_2, \ldots, \mathbf{T}_k$ are nonzero
and satisfying that
\begin{gather*}
\mathbf{T}_1 = \be_1 + \cdots + \be_{n_1}, \quad   \mathbf{T}_2 = \be_{n_1+1} + \cdots + \be_{n_2}, \quad \dots, \quad
\mathbf{T}_k = \be_{n_{k-1}+1} + \cdots + \be_{n_k},
\intertext{and}
\mathbf{T}_{k+1}=\cdots= \mathbf{T}_n={0},
\end{gather*}
where $k \le n$ and $1\le n_1 < n_2 < \cdots < n_k \leq n$.
Note that $T$ preserves $\ell_1$--norm TEA pairs exactly when $PTQ$ does.
We may replace $T$ by $PTQ$, and assume that
$T =\begin{bmatrix}
                          \mathbf{T}_1 &  \mathbf{T}_2 &\cdots & \mathbf{T}_n
                        \end{bmatrix}$.
Let $\bx = (x_1, \dots, x_n)^\TT$ and $\by = (y_1, \dots, y_n)^\TT \in \IF^n$ form a TEA pair; or equivalently,
 $\bar x_j y_j \ge 0$ for $j = 1, \dots, n$.
Then
\begin{align*}
T\bx &= x_1 T_1 + \cdots + x_k T_k = (\underbrace{x_1, \ldots, x_1}_{n_1},\,  \underbrace{x_2, \ldots, x_2}_{n_2-n_1},\, \ldots,
 \underbrace{x_k, \ldots, x_k}_{n_k -  n_{k-1}},\, 0, \ldots,0)^\TT, \\
T\by &= y_1 T_1 + \cdots + y_k T_k = (\underbrace{y_1, \ldots, y_1}_{n_1},\   \underbrace{y_2, \ldots, y_2}_{n_2-n_1},\ \ldots,
 \underbrace{y_k, \ldots, y_k}_{n_k -  n_{k-1}},\  0, \ldots,0)^\TT
\end{align*}
clearly   form a TEA pair.

Conversely, suppose $T$ preserves TEA pairs. Assume the contrary that
$T$ has a row with two nonzero entries.
We may replace $T$ by $PTQ$ for suitable monomial matrices $P$ and $Q$ and assume that $T$ has
the $(1,1)$th and the $(1,2)$th entries
equal to $1$.
Since $\bx = (2,-1, 0, \dots, 0)^\TT$ and $\by=(1,-2, 0, \dots, 0)^\TT$ form a TEA pair,   so do
$T\bx$ and $T\by$. But the first entries of $T\bx$ and $T\by$ are $1$ and $-1$, respectively. So, $T\bx$ and $T\by$ do not form a TEA pair,
a desired contradiction.

(b) If each row of $T$ has at most one nonzero entry,
then $T$ will preserve TEA pairs. Hence, $T$ will also preserve parallel pairs.
On the other hand, if $T = \bv\bu^\TT$ for some $\bu,\bv \in \IF^n$, then
$T\bx = (\bu^\TT\bx)\bv$  and $T\by = (\bu^\TT\by)\bv$ are both  scalar multiples of $\bv$, and thus  always parallel, for any
$\bx, \by \in \IF^n$.

Conversely, let $T$ be a linear parallel pair preserver with rank  larger than $1$.
We will show that every row of $T$ has at most one
nonzero entry.
Suppose on contrary that $T$ has a row with more than one nonzero entries. We may replace $T$ by $PTQ$
for some suitable monomial matrices $P$ and $Q$ and assume that the first row of $T$ has the maximum
number of nonzero entries among all the rows.
Moreover, we may also assume that all  these nonzero entries in the first row
are $1$ and lie  in the $(1,1)$th, $(1,2)$th, \ldots, $(1,k)$th positions.

Since $T$ has rank at least two, there is a row, say, the second row, which is not equal to a multiple of the first row.
We consider two  cases.

{\bf Case 1.} The first $k$ entries of the second row are not all equal.  We may replace $T$ by $TQ$ for a permutation matrix of the form
$Q = Q_1 \oplus I_{n-k}$ such that  the $(2,1)$  entry is nonzero and different from the $(2,2)$  entry.
Further replace  $T$ by $PT$ for an invertible diagonal matrix $P$  and assume that
the leading $2\times 2$ matrix of $T$ equals
$\begin{pmatrix} 1 & 1 \cr 1 & a \cr\end{pmatrix}$  for some $a \ne 1$.

Let $\bx = (m,\bar a, 0, \dots,0)^\TT$ and $\by = (1, m\bar a, 0, \dots, 0)^\TT$ with $m > 0$.
Then $\bx$ and $\by$ are parallel, and so are  the vectors   $T\bx$ and $T\by$.
The first two entries of $T\bx$ are $m+\bar a$ and $m+|a|^2$,
and the first two entries of $T\by$ are $1+ m\bar a$ and $1+ m|a|^2$.
The second entries of $T\bx$ and $T\by$ are always positive.
It forces $\overline{(m + \bar a)}(1+ m \bar a) = m(1+|a|^2) + a + m^2\bar{a}\geq 0$ for all $m>0$.
Consequently, $a\geq 0$.

Furthermore,
 if $m > 0$, then  $\bx = (m,-1,0,\dots, 0)^\TT$ and $\by = (1,-m,0, \dots, 0)^\TT$ are parallel, and so are
$T\bx$ and $T\by$. The first two entries of $T\bx$ are $m-1$ and $m-a$, and the first two entries of $T\by$ are $1-m$ and $1- am$.
It follows that $(m-a)(1-am)\leq 0$ for all $m > 0$ with $m \ne 1$. Since $a \ne 1$ and $a\geq 0$, we see that  $m = (1+a)/2 \ne 1$
and  $(m-a)(1-am) = \frac{1}{4}(1-a)(2-a-a^2)=\frac{1}{4}(1-a)^2(2+a)>0$,
which is a contradiction.

{\bf Case 2.} The first $k$ entries of the (nonzero) second row of $T$ are the same scalar $\gamma$.
If $\gamma\neq 0$ then all other entries of the second row of $T$ are zeros due to the assumption that the first row of $T$
has maximal number of nonzero entries among all rows of $T$.  But then the second row is $\gamma$ times the first row, a
contradiction. Hence, $\gamma=0$.
 Suppose the $({2},j)$th entry of $T$ equals $a \ne 0$ for some $j > k$.
We may replace $T$ by $TQ$ for a suitable permutation matrix $Q$ and assume that the leading $2\times 3$ matrix
of $T$ is $\begin{pmatrix} 1 & 1  & 0 \cr 0 & 0 & a \cr\end{pmatrix}$. Let
$\bx = (2,-1,1,0,\dots,0)^\TT$ and $\by =(1,-2,1,0,\dots,0)^\TT$.
Then $\bx$ and $\by$ are parallel and so are $T\bx$ and $T\by$.
Now, $T\bx$ has the first two entries equal to $1$ and $a$,
whereas $T\by$ has the first two entries equal to $-1$ and $a$.
Thus, $T\bx$ and $T\by$ cannot be parallel, which is a contradiction.
\end{proof}

We need to establish two more lemmas to prove Theorem \ref{p-3.2old}.

\begin{lemma} \label{lem:par-1dim-F}
If $\bV$ is a subspace of $\IF^n$ such that any two elements in $\bV$ are parallel with respect to
$\|\cdot\|_\infty$, then $\dim \bV \le 1$.
\end{lemma}

\begin{proof}
We are going to show that any nonzero $\bu,\bv\in \bV$ are linearly dependent.
To this end, we may replace $(\bu,\bv)$ by $(\alpha\bu,\beta\bv)$ for some nonzero scalars $\alpha, \beta$,
and assume that $\bu,\bv$ are unit vectors with $\|\bu+\bv\|_\infty  = 2$.
Since $\|\bu+\bv\|_\infty  = 2$, we may further replace $\bu,\bv$ by $Q\bu, Q\bv$
for a suitable generalized permutation matrix $Q$ and assume that
$\bu = (1, \dots, 1, u_{k+1}, \dots, u_n)^{\TT}$ and
$\bv = (1, \dots, 1, v_{k+1}, \dots, v_n)^{\TT}$
with $|v_j|<1$ for  $j = k+1, \dots, n$, and $k\geq 1$.

By assumption,  $\bu+\bv$ and $\bu-\bv$ are parallel with respect to the $\ell_\infty$--norm.
There exists a  unimodular scalar $\alpha$ such that
$$
2+ \|\bu-\bv\|_\infty = \|\bu+\bv\|_\infty  + \|\bu-\bv\|_\infty  = \|(\bu+\bv)+\alpha(\bu-\bv)\|_\infty =|(u_i+v_i)+\alpha(u_i-v_i)| $$
for some $i$.
 If $k+1\le i \le n$, then $|(u_i+v_i)+\alpha(u_i-v_i)|< 2+\|\bu-\bv\|_\infty$, a contradiction.
Consequently,  $1\le i\le k$.
Then $|(u_i+v_i)+\alpha(u_i-v_i)|=2$, and thus  $\bu=\bv$.
\end{proof}

\begin{lemma}\label{lemma:tea-para-not-inv-F}
Let $T: (\IF^n, \|\cdot\|_\infty) \rightarrow (\IF^n, \|\cdot\|_\infty)$
be a nonzero linear map. Then $T$ is invertible if one of the following holds.
\begin{itemize}
\item[(a)] $\IF^n\neq \IR^2$ and $T$ preserves TEA pairs,
\item[(b)] $T$ preserves parallel pairs with range space of dimension larger than one.
\end{itemize}
\end{lemma}

\begin{proof}
Recall that $\be_j$  denotes the coordinate vector with the $j$th coordinate $1$ and all others $0$ for $j=1,\ldots, n$,
and let $\be=\sum_{j=1}^{n} \be_j$ be the constant one vector in $\IF^n$.  If $J$ is a subset of $N=\{1,\ldots, n\}$,
let $\be_J = \sum_{j\in J} \be_j$.  In particular, $\be=\be_N$.

(a) Suppose $T\bv=0$ for some unit vector $\bv= \sum_{j=1}^{n} v_j \be_j \in \IF^n$.  We claim that $T$ is a zero map.
Replacing $T$ by $TR$ for a suitable general permutation
 matrix $R$, we can assume that $v_1=1 \geq v_2 \geq \cdots \geq v_n\geq 0$.

Suppose first that $v_1 = \cdots = v_k =1  > v_{k+1} \geq \cdots \geq v_n\geq 0$ for some $k < n$.
Since $\xi \bv + \be_n$ and $\xi \bv - \be_n$ form a TEA pair for large $\xi>0$, so do $T(\xi \bv + \be_n) = T(\be_n)$ and
$T(\xi \bv - \be_n)=-T(\be_n)$.  It follows $0=\|T(\be_n) + T(-\be_n)\|_\infty =\|T(\be_n)\|_\infty + \|T(-\be_n)\|_\infty$, and thus
$T(\be_n)=0$.  With $\be_n$ taking the role of $\bv$, we see that $T\be_j=0$ for all $j=1,\ldots, n-1$, and thus $T=0$.

Suppose next that $\bv=\be$ and $n\geq 3$.  For any subset $J, K$ of $N$ such that
$J\cup K\neq N$, we have
$\be-\be_J$ and $\be-\be_K$ form a TEA pair, and so do $T(\be-\be_J)=-T(\be_J)$ and $T(\be-\be_K) =-T(\be_K)$.
Hence,
$$
\|T(\be_J) + T(\be_K)\|_\infty = \|T(\be_J)\|_\infty + \|T(\be_K)\|_\infty.
$$
An inductive argument with the fact $T(\be)=0$ gives
$$
\|T(\be_1)\|_\infty = \|T(\be_2) + \cdots + T(\be_{n-1})\|_\infty = \|T(\be_2)\|_\infty + \cdots + \|T(\be_{n-1})\|_\infty.
$$
We also have similar equalities for all other $\|T(\be_j)\|_\infty$.  Summing up these $n\geq 3$ equalities, we have
$$
\sum_{j=1}^{n} \|T(\be_j)\|_\infty = (n-1) \sum_{j=1}^{n} \|T(\be_j)\|_\infty,
$$
and thus all $\|T(\be_j)\|_\infty=0$.  This also forces $T=0$.

Finally, suppose $n=2$, $\IF=\IC$ and $\bv=\be=\be_1 + \be_2$.  Since the vectors
$\be-\frac{ 1+\sqrt{3}i }{2}\be_1$
and $\be -\frac{ 1-\sqrt{3}i} {2}\be_1$
attain the triangle equality,
so do $T(\be- \frac{ 1+\sqrt{3}i }{2}\be_{1})= -\frac{ 1+\sqrt{3}i }{2}T(\be_{1})$ and
$T(\be-\frac{ 1-\sqrt{3}i }{2}\be_{ 1})=-\frac{ 1-\sqrt{3}i }{2}T(\be_{1 })$,
which implies $T(\be_{1 })=0$.  Consequently, $T(\be_2)=0$, and thus $T=0$ again.

In conclusion, a nonzero linear map $T$ preserving TEA pairs is invertible  unless $\IF^n=\IR^2$.

(b) If $n = 2$ and the range space of $T$ has dimension larger than one, then
$T$ is invertible. Suppose $n > 2$ and $T$ is not invertible.
Let $\bv$ be a nonzero vector such that $T\bv = 0$.
We may replace $T$ by the map $\bx \mapsto T(\alpha Q\bx)$
for some $\alpha > 0$ and generalized permutation matrix $Q$, and assume that
$\bv = (v_1, \dots, v_n)^\TT$ with $v_1 = \cdots = v_k = 1 >  v_{k+1}  \ge \cdots \ge  v_n \geq 0$ where $1\leq k \leq n$.
We are verifying that the range space of $T$ has dimension at most one.

{\bf Case 1.} Suppose $k = 1$. Then  for any $\bx = (0, x_2, \dots, x_n)^\TT, \by = (0, y_2, \dots, y_n)^\TT$
in the linear span $E$ of $\be_2,\ldots, \be_n$,
the vectors $r\bv + \bx, r\bv+\by$ are parallel   for sufficiently large $r > 0$.
Consequently,  $T(r\bv+\bx)=T(\bx)$ and $T(r\bv + \by) = T(\by)$ are also parallel.
By Lemma \ref{lem:par-1dim-F},
the space $T(E)$ has dimension at most one.
Since $\IF^n$ is spanned by $\bv$ and $E$ and $T\bv=0$, we conclude that $T(\IF^n)$ has dimension at most one.

{\bf Case 2.}
Suppose   $\bv =\be_1+\be_2$. In view of Case 1, we may assume $T(\be_{1})=-T(\be_{2})\neq 0$.
We claim that  $T(\be_{1})$ and $T(\bu)$ are linearly dependent for any
norm one vector $\bu=(0,0, u_3, \ldots, u_n)^\TT$.
Consequently, being the span of $T(\be_1), T(\be_2)$, and all such $T(\bu)$, the range space $T(\IF^n)$ has dimension at most one.

  Consider $\bx=\alpha \be_1+\beta \bu$ and $\by=\gamma \be_1+\delta \bu$  for any scalars $\alpha, \beta, \gamma$
  and $\delta$.
  If $\bx,\by$ are parallel, so are $T(\bx), T(\by)$.
  If $\bx$ is not parallel with $\by$ then we can assume that $\bx= \be_1 + \beta \bu$ and $\by= \gamma \be_1 + \bu$
  with $|\beta|, |\gamma|<1$.   If $\gamma\neq 0$ then
  $\bx$ is parallel with   $s\gamma \bv+ \by$ for  $s \geq \frac{1}{|\gamma|}-1$.
   We see that $T(\bx)$ is parallel with $T(s\gamma \bv + \by) = T(\by)$.
In case when $\gamma=0$, we see that $T(\bx)$ is parallel with
  $ T(s\epsilon \bv+\epsilon \be_1+ \bu)= T(\epsilon \be_1+ \bu)=\epsilon T(\be_1)+T(\by)$ whenever $0<\epsilon<1$
   and $s \geq \frac{1}{\epsilon}-1$.
  In other words,
  $$
  \|T(\bx) + \mu_\epsilon(\epsilon T(\be_1)+T(\by))\|_\infty = \|T(\bx)\|_\infty  + \|\epsilon T(\be_1)+T(\by)\|_\infty
  $$
  for some unimodular scalar $\mu_\epsilon$.
  Choosing a sequence
   $\epsilon_n \to 0^+$ with $\mu_{\epsilon_n}$ converging to some unimodular $\mu$,  we see
  that $T(\bx)$ and $T(\by)$ are parallel.
  Therefore, in any case $T(\bx)$ is parallel with $T(\by)$. Hence $T(\be_{1})$ and $T(\bu)$
  are linearly dependent by Lemma \ref{lem:par-1dim-F}, as claimed.

{\bf Case 3.}   Suppose that
 $\bv=  \be_{1} + \be_{2} + a_3 \be_{3} + \cdots +  a_n \be_{n}$ with
  $1\geq  a_3 \geq a_4 \geq    \cdots \geq a_m \geq 0$.
In view of Case 2, we can assume that $a_3 > 0$.

Consider $\bx=\alpha \be_1+\beta \be_2$ and $\by=\gamma \be_1+\delta \be_2$. We claim that $T(\bx)$ and $T(\by)$ are parallel.
If $\bx, \by$ are parallel, then it is the case. Otherwise, we can assume that $\bx=\be_1 + \mu \be_2$ and $\by=\nu \be_1 + \be_2$ with $|\mu|, |\nu|<1$.
If $|1-\mu| \geq |1+\mu|a_3$, then $\bx - (1+\mu)\bv/2$ and $\by$ are parallel, and so are $T(\bx)=T(\bx - (1+\mu)\bv/2)$ and $T(\by)$.
If $|1-\nu| \geq |1+\nu|a_3$, then $\bx$ and $\by - (1+\nu)\bv/2$ are parallel, and so are $T(\bx)$ and $T(\by)=T(\by - (1+\nu)\bv/2)$.
If $|1-\mu| < |1+\mu|a_3$ and $|1-\nu| < |1+\nu|a_3$ then $\bx - (1+\mu)\bv/2$ and
$\by - (1+\nu)\bv/2$ are parallel, then so are $T(\bx)=T(\bx - (1+\mu)\bv)$
and $T(\by)=T(\by - (1+\nu)\bv)$.
 We   see that $T(\bx)$ and $T(\by)$ are parallel in all cases.
 Hence $T(\be_{1})$ and $T(\be_{2})$ are linearly dependent by Lemma \ref{lem:par-1dim-F}.
Consequently, there is a nontrivial linear combination $\bv'=\xi \be_{1} + \eta \be_{2}$ belongs to the kernel of $T$.
We can then reduce the situation to either Case 1 (if $|\xi|>|\eta|$) or Case 2 (if $|\xi|=|\eta|$).

We thus conclude that a  linear parallel pair preserver is invertible if its range space has dimension larger than one.
\end{proof}

Consider the linear map $T: \IR^2\to \IR^2$ defined by $(\bx, \by)^\TT\mapsto (\bx-\by,0)^\TT$.  It is easy
to see that $T(\IR^2)$ has dimension one and
$T$ preserves TEA/parallel pairs for the $\ell_\infty$--norm.  This example says that Lemma \ref{lemma:tea-para-not-inv-F}
does not hold in the missing cases.

\begin{lemma} \label{A-matrix}
Let $A = (a_{rs}) \in \bM_n$.
Suppose either
 $a_{jj} > |a_{jk}|$ whenever $j \ne k$, or
  $a_{jj} > |a_{kj}|$ whenever $j \ne k$.
The following conditions are equivalent.
\begin{itemize}
\item [{\rm (a)}]
For any
 $\bx = (x_1, \dots, x_n)^{\TT}\in \IF^n$ with $\by = A^{\TT}\bx = (y_1, \dots, y_n)^{\TT}$, we have
\begin{align}\label{eq:D>}
\|\bx\|_\infty = |x_r| > |x_s|\ \text{whenever  $s \ne r$}\quad
\implies\quad \|\by\|_\infty = |y_r| \geq |y_s|\ \text{whenever $s\ne r$.}
\end{align}
\item[{\rm (b)}] Either $n = 2$ and $A = A^*$ with $a_{11}=a_{22}>|a_{12}|$, or   $A = a_{11}I_n$.
\end{itemize}
\end{lemma}

\begin{proof} The implication (b) $\Rightarrow$ (a) is clear if $A=a_{11}I_n$.
Suppose $n = 2$ and $A = A^*$ with $a_{11}=a_{22}>|a_{12}|$.
Observe that
\begin{align*}
\begin{pmatrix} y_1\\ y_2\end{pmatrix} = A^{\TT} \begin{pmatrix} x_1 \\ x_2 \end{pmatrix}
= \begin{pmatrix} a_{11} & \overline{a_{12}} \\ a_{12} & a_{11}\end{pmatrix}
\begin{pmatrix} x_1 \\ x_2 \end{pmatrix}
= \begin{pmatrix} a_{11} x_1 +  \overline{a_{12}} x_2 \\ {a_{12}} x_1 + a_{11} x_2\end{pmatrix},
\end{align*}
and
\begin{align*}
& |y_1| \geq |y_2| \\
\Longleftrightarrow\quad
& |a_{11} x_1 +  \overline{a_{12}} x_2| \geq | {a_{12}} x_1 + a_{11} x_2| \\
\Longleftrightarrow\quad
& a_{11}^2|x_1|^2 + |a_{12}|^2|x_2|^2 + 2a_{11} \operatorname{Re}\, ( {a_{12}} x_1 \bar x_2)
\geq
|a_{12}|^2 |x_1|^2 + a_{11}^2|x_2|^2 + 2a_{11} \operatorname{Re}\,  ({a_{12}} x_1 \bar x_2) \\
\Longleftrightarrow\quad
& (a_{11}^2 - |a_{12}|^2)|x_1|^2 \geq (a_{11}^2 - |a_{12}|^2)|x_2|^2 \\
\Longleftrightarrow\quad
& |x_1| \geq |x_2|.
\end{align*}
Consequently, (b) $\Rightarrow$ (a) also holds in this case.

We are going to  prove (a) $\Rightarrow$ (b).
We may replace $A$ by $P^{\TT}AP$ with a suitable
permutation matrix $P$ and assume the first row $\bv_1$ of $A^{\TT}$ has the maximal $\ell_1$--norm.
Then further replace $A$ by $DA D^*$ with a suitable diagonal  $D \in \bM_n$ with $D^*D= I_n$
and assume that the first row of $A^{\TT}$ has nonnegative entries.

By the continuity and an induction argument,  for $\bx = (x_1, \dots, x_n)^{\TT}$, $\by = A^{\TT} \bx = (y_1, \dots, y_n)^{\TT}$
and $k=1,\ldots, m$, we have
\begin{align}\label{eq:kequal}
|x_1| = \cdots = |x_k|> |x_r|\ \text{for all $r>k$}\quad\implies\quad
 |y_1| = \cdots = |y_k|\geq |y_r| \ \text{for all $r>k$.}
\end{align}
Let  $\bx = (1,\dots, 1)^{\TT}$ and ${\by}=A^{\TT}\bx$. If $\bv_1, \dots, \bv_n$ are the rows of
$A^{\TT}$, then
\begin{align*}
\|\bv_1\|_1 &= a_{11} + \cdots + a_{n1} =
|y_1| = |y_j| =
|a_{1j} + \cdots + a_{nj}| \\
&\le |a_{1j}| + \cdots + |a_{nj}|
= \|\bv_j\|_1 \le \|\bv_1\|_1\quad\text{for $j \geq 1$.}
\end{align*}
Since $a_{jj} > 0$, we see that $a_{ij} \ge 0$ for all $i\ne j$,
and all row sums of $A$ are equal, say, to $s>0$.

Taking $\bx = (1,1, 0,\ldots,0)^{\TT}$, with \eqref{eq:kequal} we have $ a_{11}+a_{21}=a_{12}+a_{22}$.
Similarly, taking $\bx = (1,-1, 0,\ldots,0)^{\TT}$, we have $|a_{11}- a_{21}| = |a_{12} - a_{22}|$.
It follows from either the assumption  $a_{11}> a_{12}$ and $a_{22}> a_{21}$, or the assumption
 $a_{11}> a_{21}$ and $a_{22}> a_{12}$ that
$a_{11} - a_{21} =  a_{22}-a_{12}$.
Consequently, $a_{11} = a_{22}$
and $a_{12} = a_{21}$.
The assertion follows when $n=2$.

Suppose $n\geq 3$.
Apply the same argument to other pairs $(i,j)$ with $i\neq j$ instead of
$(1,2)$, we see that
$$
a_{11}=\cdots = a_{nn}\quad\text{and}\quad a_{jk}=a_{kj}\quad\text{whenever $j\neq k$.}
$$
For a fixed $j = 1, \dots, n$, we take $\bx = (1,\ldots, 1, \underbrace{-1}_{j\text{th}}, 1,\ldots, 1)^{\TT}$
and ${\by}=A^{\TT}\bx$.
For   distinct indices $j,k, l$, we have
$|y_{l}| = s -2 a_{j{l}}=|y_k| = s - 2a_{jk}$.
It follows $a_{j{l}} = a_{jk}=a_{kj}$, and thus
$a_{jk}=a_{12}$ are all equal for $j \ne k$.
Consider $\bu = (1, -1, \ldots, -1)^{\TT}$ and ${\bv}=A^{\TT}\bu= (v_1, v_2,  \ldots, v_n)^{\TT}$.
Then $|v_1|=|v_2|$ implies either
$$
a_{11} -  (n-1)a_{12}= a_{11} + (n-3) a_{12}\quad\text{or}\quad (n-1)a_{12}-a_{11}= a_{11} + (n-3)a_{12}.
$$
  Since $n\geq 3$,  either
$a_{12}=0$ or $a_{11}=a_{12}$.  But $a_{11}> a_{12}$.  This implies that $A=a_{11}I_n$.
\end{proof}

We are now ready to present the

\begin{proof}[Proof of Theorem \ref{p-3.2old}]
(a) It is clear that $T$ preserves parallel pairs  if $T$ has the form in  (a.1) or (a.3).
Suppose $n = 2$ and $T$ has the form in (a.2).
Then, $T$ preserves parallel pairs in $(\IF^2, \|\cdot\|_\infty)$
by the implication from (b) to (a) in
Lemma \ref{A-matrix};  indeed, with a continuity argument the condition \eqref{eq:D>} implies that
$$
\|\bx\|_\infty = |x_r| \geq |x_s|\ \text{whenever  $s \ne r$}\quad
\implies\quad \|\by\|_\infty = |y_r| \geq |y_s|\ \text{whenever $s\ne r$.}
$$

Conversely, suppose $T$
preserves parallel pairs.
If $T$ is not invertible then it follows from  Lemma \ref{lemma:tea-para-not-inv-F} that
 $T$  is either the zero map or has the form in (a.3).
Suppose from now on $T$ is invertible.

 If $T^{-1}\be_j ={\bx}_j$, then ${\bx}_i$ and ${\bx}_j$ cannot be
parallel for any $i\ne j$. Otherwise, $T{\bx}_i = \be_i$ and $T{\bx}_j = \be_j$
were parallel. Thus, the vectors ${\bx}_i$ and ${\bx}_j$ cannot attain the $\ell_\infty$-norm
at a same coordinate. So, there is a permutation $\sigma$ on $\{1,\ldots, n\}$ such that
each ${\bx}_j$ attains its norm at its $\sigma(j)$th coordinate but no other.
Consequently, there is a generalized permutation matrix
$Q \in \bM_n$ such that the map $L$ defined by
${z}\mapsto QT^{-1}{z}$ will send $\be_j$ to a vector ${\by}_j = (a_{j1}, \dots, a_{jn})^{\TT}$
such that $a_{jj} > |a_{ji}|$ for all $i\ne j$. Clearly, $L^{-1}$  defined by
${\by} \mapsto T(Q^{-1}{\by})$ preserves parallel pairs.

Let $A = (a_{ij}) \in \bM_n$ so that $L({\bx}) = A^{\TT} {\bx}$.
If ${\bx}_j=(x_{j1}, \ldots, x_{jn})^{\TT}$ satisfies
 $|{x}_{jj}| > |{x}_{ji}|$ for all $i\ne j$, then we claim that $L({\bx}_j) = A^{\TT} {\bx}_j$ is parallel to
${\by}_j$ and thus $\be_j$, but not any other ${\by}_i$.  Otherwise, the fact
 $L({\bx}_j)$ is parallel to ${\by}_i$ for some $i \ne j$ would imply that
${\bx}_j$ is parallel to $L^{-1}({\by}_i) = \be_i$, which is impossible.
Thus, the matrix  $A = (a_{ij})$ satisfies the hypothesis (a)
of Lemma \ref{A-matrix}. If $n\geq 3$,
 we see that $A = \gamma I_n$.
with $\gamma = a_{11}>0$.
If $n = 2$, we see that
$A$ is Hermitian with $a_{11}=a_{22}$. So, $T$ has the form in (a.2).

(b) It is clear that if $T$ assumes the form in (b.1) then it preserves TEA pairs.
It is also the case if $T$ assumes the form in (b.3) by direct verification.
Suppose $\IF^n=\IR^2$ and $T = \gamma CQ$ as in (a.2).
 Clearly, $T$ satisfies (\ref{g-te-pre}) if and only if
$C$ satisfies (\ref{g-te-pre}). There is
$S = \diag(1, \pm1)$ such that $SCS$
has the form  $\hat C = \begin{pmatrix} 1 & \beta \cr
\beta & 1 \cr \end{pmatrix}$ with $0\leq  \beta < 1$.
The map
${\bx} \mapsto C{\bx}$ satisfies (\ref{g-te-pre})
if and only if the map ${\bx} \mapsto \hat C {\bx}$ does.
Now, if the nonzero vectors $\bx,\by$ in $\IR^2$ satisfy that
 $\|{\bx}+{\by}\|_\infty = \|{\bx}\|_\infty + \|{\by}\|_\infty$, then we may
replace $({\bx},{\by})$ by $ (\xi{\bx}/\|\bx\|_\infty,\xi{\by}/\|\by\|_\infty)$ with $\xi \in \{-1,1\}$
and assume that ${\bx} = (1,x_2)^{\TT}, {\by} = (1,y_2)^{\TT}$ with $|x_2|, |y_2| \le 1$, or
${\bx} = (x_1, 1)^{\TT}, {\by} = (y_1, 1)^{\TT}$ with $|x_1|, |y_1| \le 1$.
One can  check that
$\|\hat C{\bx} + \hat C{\by}\|_\infty = \|\hat C{\bx}\|_\infty + \|\hat C{\by}\|_\infty$.
Thus, $T$ preserves TEA as well.

Conversely, suppose the map $T$ is nonzero and satisfies (\ref{g-te-pre}). Then $T$ will preserve
parallel pairs.   Thus, it will be of the form (a.1), (a.2), or (a.3).
We will show that (a.2) is impossible in the complex case unless it reduces to the form in (a.1),
and there are additional restrictions
for ${\bu}$ if (a.3) holds.

Suppose $\IF^n = \IC^2$ and $T = \gamma CQ$ has the form in (a.2)
in which $C=\begin{pmatrix} 1 & \beta \\ \overline{\beta} & 1\end{pmatrix}$
for some complex scalar $\beta$ with  $|\beta|<1$.  In this case, the map $\bx\mapsto Cx$ also preserves TEA pairs.
Consider  ${\bx} = (1,0)^{\TT}$, ${\by} = (1,1)^{\TT}$ and ${z} = (1,i)^{\TT}$, and their images
 $Cx = (1, \overline{\beta})^\TT$, $Cy=(1+  \beta, \overline{\beta}+1)^\TT$
 and $Cz= (1+ i\beta, \overline{\beta}+i)^\TT$.
Note that  $\bx,\by$ and $\bx,\bz$ are both TEA pairs,  while $Cx, Cy$ and $Cx, Cz$ form   TEA pairs exactly
when the first coordinates of $Cy$ and $Cz$
 assume  positive values.  This forces $\beta=0$, and thus $T=\gamma Q$ reduces to the form in (a.1).

Suppose $T=\bv{\bu}^\TT$  assumes the form in (a.3) for some vectors $\bu=(u_1, u_2)^\TT$ and $\bv$ in $\IF^n$.
In this case, $T$ is not invertible.
By Lemma \ref{lemma:tea-para-not-inv-F}, $\IF^n= \IR^2$.
Since ${y}_1 = (1,1)^{\TT}$ and ${\by}_2 = (1,-1)^{\TT}$ form a TEA pair, so are $T\by_1 = (u_1 + u_2)\bv$ and $T\by_2 = (u_1-u_2)\bv$.
This forces $u_1 + u_2$ and $u_1 - u_2$ have the same sign.  Similarly, ${\by}_1$ and $-{\by}_2$ also form a TEA pair, and thus
$u_1 + u_2$ and $-u_1 + u_2$ also have the same sign.
If $u_1+u_2\neq 0$ then $u_2 = u_1$.  In any case, we have $|u_1|=|u_2|$ as asserted.
\end{proof}

\section{The infinite dimensional cases}\label{s:inf-dim}

 Let  the underlying field $\IF$ be either $\IR$ or $\IC$, and
let $\Lambda$ be a finite or an infinite index set.  When $1\leq p < \infty$, let $\ell_p(\Lambda)$ be the (real or complex) Banach space
of   $p$-summable families $\bx=(x_\lambda)_{\lambda\in \Lambda}$ (of real or complex numbers) with $\ell_p$--norm
$$
\|\bx\|_p = \left(\sum_{\lambda\in \Lambda} |x_\lambda|^p\right)^{{1}/{p}} < +\infty.
$$
Note that the above sum is finite only if there are at most countably many coordinates $x_\lambda\neq 0$.
Let $\ell_\infty(\Lambda)$ be the Banach space of  uniformly bounded family $\bx=(x_\lambda)_{\lambda\in \Lambda}$
with the $\ell_\infty$--norm
$$
\|\bx\|_\infty = \sup_{\lambda\in \Lambda} |x_\lambda| < +\infty.
$$
We are also interested in the Banach  subspace $c_0(\Lambda)$ of $\ell_\infty(\Lambda)$
consisting of essentially null families $\bx=(x_\lambda)_{\lambda\in \Lambda}$ for which
for any $\epsilon >0$ there are at most finitely
many coordinates $x_\lambda$ with $|x_\lambda|\geq \epsilon$.
It is plain that the vector spaces satisfying
$$
\ell_1(\Lambda) \subseteq \ell_p(\Lambda) \subseteq c_0(\Lambda)\subseteq \ell_\infty(\Lambda),
\quad\text{whenever $1< p< \infty$}.
$$
When $\Lambda$ is a finite set, all above spaces coincide; otherwise, all inclusions are proper.
We write $\ell_p$ and $c_0$ for $\ell_p(\mathbb{N})$ and $c_0(\mathbb{N})$ as usual.

As in the finite dimensional case, the $\ell_p$--norm is strictly convex when $1<p<\infty$.
Two nonzero
$\bx, \by$ in $\ell_p(\Lambda)$ are parallel (resp.\ TEA), exactly when there is a scalar $t$ (resp.\ $t>0$) such
that $\bx=t\by$.  Therefore,
any linear map of $\ell_p(\Lambda)$ preserves parallel pairs and TEA pairs when $1<p<\infty$.

We study the cases when $p=1$ and $p=\infty$, and $\Lambda$ is an \emph{infinite} index set below.
As in Section \ref{s:finite-dim}, we start with the following observations.
Note that $\ell_\infty(\Lambda)$ is isometrically isomorphic to
the Banach space $C(\beta\Lambda)$ of continuous functions on the Stone-Cech compactification
$\beta\Lambda$ of $\Lambda$, which consists of all ultrafilters of the discrete space $\Lambda$.

\begin{lemma}\label{lem:inf-seqn-para}
\begin{enumerate}[{\rm (a)}]
  \item  $\bx=(x_\lambda)_{\lambda\in \Lambda}, \by=(y_\lambda)_{\lambda\in \Lambda}$ in $\ell_1(\Lambda)$ are parallel (resp.\ TEA)  with respect to the $\ell_1$--norm if and only if
  there is a unimodular scalar $\mu$ (resp.\ $\mu=1$) such that $\mu \overline{x_\lambda}y_\lambda\geq 0$ for all $\lambda\in \Lambda$.

  \item  $\bx=(x_\lambda)_{\lambda\in \Lambda}, \by=(y_\lambda)_{\lambda\in \Lambda}$  in $c_0(\Lambda)$
    are parallel (resp.\ TEA) with respect to the $\ell_\infty$--norm if and only if there is an index  $\lambda$ in $\Lambda$ such that
       $|\overline{x_\lambda}y_\lambda| = \|\bx\|_\infty \|\by\|_\infty$
(resp.\   $\overline{x_\lambda}y_\lambda = \|\bx\|_\infty \|\by\|_\infty$).

  \item  $\bx=(x_\lambda)_{\lambda\in \Lambda}, \by=(y_\lambda)_{\lambda\in \Lambda}$  in $\ell_\infty(\Lambda)$
    are parallel (resp.\ TEA) with respect to the $\ell_\infty$--norm if and only if there is an ultrafilter $\fU$ on $\Lambda$ such that
       $\lim_{\fU} |\overline{x_\lambda}y_\lambda| = \|\bx\|_\infty \|\by\|_\infty$
(resp.\   $\lim_{\fU} \overline{x_\lambda}y_\lambda = \|\bx\|_\infty \|\by\|_\infty$).
  \end{enumerate}
\end{lemma}

For each $\alpha\in \Lambda$, let $\be_\alpha= (e_{\alpha, \lambda})_{\lambda\in \Lambda}$
be the $\alpha$--coordinate vector with coordinates $e_{\alpha, \lambda}=1$ when $\alpha=\lambda$ and $0$ elsewhere.
We can identify any bounded linear map $T$ of $\ell_1(\Lambda)$ or $c_0(\Lambda)$ as the infinite ``matrix'' $(t_{\alpha\beta})$ with
$t_{\alpha\beta} = \be_\alpha^\TT(T\be_\beta)$, where $\be_\alpha^\TT$ denotes the linear functional
$(x_\lambda)_{\lambda\in \Lambda}\mapsto x_\alpha$ for any $\alpha\in \Lambda$.
However, there are unbounded linear maps   such that their representation ``matrices'' are zero.
For example, consider any unbounded linear functional $f$ of $c_0$ vanishing on the subspace of all finite sequences, that is,
$f(\be_n)=0$ for all $n=1,2,\ldots$.  Then the unbounded linear map $\bx\mapsto f(\bx)\be_1$ has zero ``matrix''.

On the other hand, the Banach dual space of $\ell_1(\Lambda)$ is
$\ell_\infty(\Lambda)$ when one identify $\bu=(u_\lambda)_{\lambda\in\Lambda}\in \ell_\infty(\Lambda)$
with the bounded linear functional $\bu^\TT= \sum_{\lambda\in\Lambda} u_\lambda\be_\lambda^\TT$ (converging in the weak* topology
$\sigma(\ell_\infty(\Lambda),\ell_1(\Lambda)$).
In this case, a nonzero bounded linear operator $S$ of $\ell_\infty(\Lambda)$ can have
zero representation ``matrix''.  For example, let $g$ be any nonzero bounded linear functional of $\ell_\infty$ vanishing on the
essential null sequence space $c_0$, and $S\bx = g(\bx)\be_1$.  However, a
$\sigma(\ell_\infty(\Lambda),\ell_1(\Lambda))$--$\sigma(\ell_\infty(\Lambda),\ell_1(\Lambda))$ continuous linear map is determined by
its representation ``matrix''.

Using the terminology in the finite dimensional case, we call a ``matrix'' $U= (u_{\alpha\beta})$
 a ``\emph{monomial matrix}'' if for
each $\alpha\in \Lambda$ there is exactly one $\beta\in \Lambda$ such that $u_{\alpha\beta}\neq 0$.
A ``monomial matrix'' $U$ is a ``\emph{generalized permutation matrix}'' if all its nonzero entries $|u_{\alpha\beta}|=1$, and it is a
``\emph{diagonal unitary matrix}''  if $|u_{\alpha\alpha}|=1$ for
each $\alpha\in \Lambda$.
We also assume that the linear map $U$ is bounded or
$\sigma(\ell_\infty(\Lambda),\ell_1(\Lambda))$--$\sigma(\ell_\infty(\Lambda),\ell_1(\Lambda))$ continuous,
depending on the context, so that
the representation ``\emph{monomial matrix}'' $(u_{\alpha\beta})$ determines $U$.
It is   clear that a  linear map $T$ of $\ell_1(\Lambda)$ (resp.\ $c_0(\Lambda)$ or $\ell_\infty(\Lambda)$)
 preserves parallel pairs or TEA pairs if, and only if,
$\gamma UTV$ does whenever $\gamma>0$, and $U,V$ are some invertible
``monomial matrices'' (resp.\ ``generalized permutation matrices'').

\begin{theorem}
\label{thm:inf-1-preservers}
Let $T=(t_{\alpha\beta}): \ell_1(\Lambda)\to \ell_1(\Lambda)$ be a bounded linear map.
\begin{itemize}
\item[{\rm (a)}] $T$ preserves TEA pairs if and only if for each $\alpha\in \Lambda$ there is
at most one $\beta\in \Lambda$ such that  $t_{\alpha\beta}\neq 0$.
\item[{\rm (b)}] $T$ preserves parallel pairs if and only if $T$ preserves TEA,
or $T = {\bv}{\bu}^\TT$ for some   $\bu\in \ell_\infty(\Lambda)$ and $\bv  \in \ell_1(\Lambda)$.
\end{itemize}
\end{theorem}
\begin{proof}
(a) Suppose for each $\alpha\in \Lambda$ there is
at most one $\beta\in \Lambda$ such that $t_{\alpha\beta}\neq 0$.  Write such $\beta=\alpha'$ in this case.
Let  $\bx=(x_\lambda)_{\lambda\in \Lambda}, \by=(y_\lambda)_{\lambda\in \Lambda}$ be a TEA pair in $\ell_1(\Lambda)$.
By Lemma \ref{lem:inf-seqn-para},
 $\overline{x_\lambda}y_\lambda\geq 0$ for all $\lambda\in \Lambda$.
For each $\lambda\in \Lambda$, the $\lambda$-coordinate of  $T\bx$ and $T\by$ are $t_{\lambda\lambda'}x_{\lambda'}$
and $t_{\lambda\lambda'}y_{\lambda'}$, respectively.  Since
$\overline{t_{\lambda\lambda'}x_{\lambda'}}t_{\lambda\lambda'}y_{\lambda'}=
\overline{x_{\lambda'}}y_{\lambda'} |t_{\lambda\lambda'}|^2\geq 0$ for all $\lambda\in \Lambda$, we see that $T\bx, T\by$
form a TEA pair.
 The converse  follows from exactly the same arguments for the finite dimensional case given in the proof of
 Theorem  \ref{prop:1-preservers}(a).

 (b) It suffices to verify the necessity for the case when $T$ does not preserve TEA.
 Suppose $T$ preservers parallel pairs, and the ${\alpha_1}$-row
 $(t_{{\alpha_1}\beta})_{\beta\in\Lambda}$ of its  matrix   representation
 has more than one nonzero entries.
 Suppose the range of $T$ has dimension at least two, and thus there is another ${\alpha_2}$--row
  $(t_{{\alpha_2}\beta})_{\beta\in\Lambda}$ of $T$ linearly independent from
 the ${\alpha_1}$-row.  We are going to derive a contradiction.

 Since the ${\alpha_1}$--row and the ${\alpha_2}$--row of $T$ are linearly independent, there are distinct indices $\beta_1, \beta_2$
 such that the $2\times 2$ matrix
$$
\begin{pmatrix}
  t_{{\alpha_1}\beta_1} & t_{{\alpha_1}\beta_2} \\
  t_{{\alpha_2}\beta_1} & t_{{\alpha_2}\beta_2}
\end{pmatrix}
$$
is invertible.
If both $t_{{\alpha_1}\beta_1}, t_{{\alpha_2}\beta_1}$, or both $t_{{\alpha_1}\beta_2}, t_{{\alpha_2}\beta_2}$, are nonzero, then by
replacing $T$ with $PTQ$ for some suitable invertible ``monomial matrix'' $P, Q$, we can assume that the above matrix
assumes the form
$$
\begin{pmatrix}
  1 & 1 \\
  1 & a
\end{pmatrix}
$$
for some scalar $a\neq 1$.
Then the argument in Case 1 of the proof of
 Theorem  \ref{prop:1-preservers}(b) derives a desired contradiction.
 If $t_{{\alpha_1}\beta_2}=t_{{\alpha_2}\beta_1}=0$, say, then we
 search for an other index $\beta_3$ with
$t_{{\alpha_1}\beta_3}\neq 0$, and  such $\beta_3$ exists by assumption.
If $t_{{\alpha_2}\beta_3}\neq 0$, then it comes back to the first case above, and we are done.
Suppose $t_{{\alpha_2}\beta_3}= 0$.  By replacing $T$ with $P'TQ'$ for some suitable invertible ``monomial matrices'' $P', Q'$, we may assume that $T$ has a ``submatrix'' of the form
$$
\begin{pmatrix}
  1 & 1 & 0 \\
  0 & 0 & a
\end{pmatrix}
$$
with $a\neq0$ for the indices ${\alpha_1}, {\alpha_2}$, and $\beta_1, \beta_2, \beta_3$.
Then the argument in Case 2 of the proof of Theorem  \ref{prop:1-preservers}(b) provides us a contradiction, as well.
\end{proof}


\begin{lemma} \label{lem:inf-par-1dim-F}
Let  $\Lambda$ be an infinite index set.
\begin{enumerate}[{\rm (a)}]
\item
If $\bV$ is a subspace of $\ell_\infty(\Lambda)$ such that any two elements in $\bV$ are parallel with respect to
$\|\cdot\|_\infty$, then $\dim \bV \le 1$.

\item
Let $T: \ell_\infty(\Lambda)\to \ell_\infty(\Lambda)$ or $T: c_0(\Lambda)\to c_0(\Lambda)$
be a nonzero linear map. Then $T$ is injective if one of the following holds.
\begin{enumerate}[{\rm (i)}]
\item    $T$ preserves TEA pairs.
\item    $T$ preserves parallel pairs with range space of dimension larger than one.
\end{enumerate}
\end{enumerate}
\end{lemma}
\begin{proof}
  (a) It follows from  the proof of Lemma \ref{lem:par-1dim-F}.

  (b) We discuss only the case $T$ is a linear map of $\ell_\infty(\Lambda)$, since the other case is similar.

   (i) Let $T$ be a  linear TEA preserver of $\ell_\infty(\Lambda)$ such that $T(\bv)=0$ for
  some  $\bv =(v_\lambda)_{\lambda\in\Lambda}$ with $\|\bv\|_\infty =\sup_{\lambda\in \Lambda} |v_\lambda|=1$.
  We will show that   $T=0$.

Suppose $\Lambda'=\{\lambda\in \Lambda: |v_\lambda|<1\}\neq \emptyset$.
For any $\lambda'\in \Lambda'$, since $\xi \bv + \be_{\lambda'}$ and  $\xi \bv - \be_{\lambda'}$ form a TEA pair for large $\xi>0$,
so do $T(\xi \bv + \be_{\lambda'})=T(\be_{\lambda'})$ and $T(\xi \bv - \be_{\lambda'})=-T(\be_{\lambda'})$.
It forces $T(\be_{\lambda'})=0$.
For any $\bu=\sum_{\lambda\neq\lambda'} v_\lambda\be_\lambda$ with zero $\lambda'$--coordinate,
we have $\xi \be_{\lambda'} + \bu$ and $\xi \be_{\lambda'} - \bu$ form a TEA pair for large $\xi>0$,
and so do $T(\xi \be_{\lambda'} + \bu)=T(\bu)$ and $T(\xi \be_{\lambda'} - \bu)=-T(\bu)$.
It follows $T(\bu)=0$.  In general, for any $\bx = x_{\lambda'}\be_{\lambda'} + \bu$
such that $\bu=\bx-x_{\lambda'}\be_{\lambda'}$ has zero ${\lambda'}$--coordinate,
$$
T(\bx)= x_{\lambda'}T(\be_{\lambda'}) + T(\bu)=0.
$$
Hence $T=0$.
Therefore, we may assume $\Lambda'=\emptyset$.

Replacing $T$ with $TQ$ for some suitable ``generalized permutation matrix'' $Q$, we may further assume that $\bv=\be_\Lambda$, that
is all coordinates of $\bv$ is $1$.
For any distinct $\lambda_1, \lambda_2\in \Lambda$, since $\be_\Lambda - \be_{\lambda_2}$
and $\be_{\lambda_1}+ \be_{\lambda_2}$
form a TEA pair, so do $T(\be_\Lambda - \be_{\lambda_2})=-T(\be_{\lambda_2})$ and
$T(\be_{\lambda_1}+ \be_{\lambda_2})=  - T(\be_\Lambda - \be_{\lambda_1}- \be_{\lambda_2})$, since $T(\be_\Lambda)=0$.
Therefore,
  $$
  \|T(\be_{\lambda_1})\|_\infty = \|T(\be_\Lambda - \be_{\lambda_1})\|_\infty=
  \|T(\be_{\lambda_2})\|_\infty +
  \|T(\be_\Lambda - \be_{\lambda_1}- \be_{\lambda_2})\|_\infty
  $$
  for any distinct $\lambda_1, \lambda_2\in \Lambda$.
  Exchanging the roles of $\lambda_1$ and $\lambda_2$, we see that
  $$
 T(\be_\Lambda - \be_{\lambda_1}- \be_{\lambda_2})   =0.
$$
Replacing $\bv=\be_\Lambda$ with $\bv'=\be_\Lambda - \be_{\lambda_1}- \be_{\lambda_2}\neq 0$ (since $\Lambda$ has more than
two elements), and arguing as above,
we see that $\Lambda'=\{\lambda_1, \lambda_2\}\neq\emptyset$,
and then $T=0$.

(ii) Suppose $T$ preserves parallel pairs.
Assume $T$ is not injective,
and $\bv =(v_\lambda)_{\lambda\in\Lambda}$ is a norm one element such that $T\bv = 0$.
Let $\Lambda'=\{\lambda\in \Lambda: |v_\lambda|<1\}$.
For any $\bx, \by\in \ell_\infty(\Lambda)$ such that their $\lambda$--coordinates $\be_\lambda^\TT\bx=\be_\lambda^\TT\by=0$
for all $\lambda$ outside $\Lambda'$,
we see that $\xi\bv + \bx$ and $\xi\bv + \by$ are parallel for large $\xi>0$, and so are
$T(\xi\bv + \bx)=T(\bx)$ and $T(\xi\bv + \by)=T(\by)$.  It follows from part (a) that the space
$\{T(\bx): \be_{\lambda}^\TT\bx= 0 \ \text{for all $\lambda$ outside $\Lambda'$}\}$
has dimension at most one.

If $\Lambda=\Lambda'$ then we are done.
If $\Lambda''=\Lambda\setminus \Lambda'$ is nonempty, then
by replacing $T$ with $TQ$ for some ``generalized permutation matrix'' $Q$, we may assume that
 $v_\lambda=1$ for all $\lambda\in \Lambda''$.
Then the proof of Lemma \ref{lemma:tea-para-not-inv-F}(b) shows that the range of  $T$
has dimension at most one.
\end{proof}

We note that unlike the finite dimensional case, an injective TEA/parallel pair linear preserver of
$c_0(\Lambda)$ or $\ell_\infty(\Lambda)$
can be non-surjective.  For an example, consider the isometric
right shift operator $L$ of $\ell_\infty$ or $c_0$ by sending $\be_n$ to $\be_{n+1}$ for
$n=1,2,\ldots$.
However, the following result demonstrates that    TEA/parallel pair linear preservers of
$c_0(\Lambda)$ or $\ell_\infty(\Lambda)$ are automatically bounded, and indeed   scalar multiples of
  injective isometries.

\begin{theorem}\label{thm:c_0}
  Let $\Lambda$ be an infinite index set.
Let  $T: c_0(\Lambda)\rightarrow c_0(\Lambda)$ be a nonzero  linear map.
The following conditions are
equivalent to each other.
 \begin{enumerate}[{\rm (a)}]
\item    $(T(\bu), T(\bv))$   is a   TEA pair if and only if $(\bu, \bv)$ is a TEA pair, for any $\bu, \bv\in c_0(\Lambda)$.
\item    $(T(\bu), T(\bv))$   is a   parallel pair if and only if $(\bu, \bv)$ is a parallel pair, for any $\bu, \bv\in c_0(\Lambda)$.
\item  $T$ is a scalar multiple of a (not necessarily surjective) linear isometry.
\end{enumerate}
In this case, there is $\gamma > 0$, a subset $\Lambda_1$ of $\Lambda$,
a family $\{\mu_\lambda: \lambda\in \Lambda_1\}$ of unimodular scalars,
 and a surjective  map $\tau: \Lambda_1\to \Lambda$ such that
for any $\bx=(x_\lambda)_{\lambda\in \Lambda} \in c_0(\Lambda)$, the image
$\by=T(\bx)=(y_\lambda)_{\lambda\in \Lambda}$
has coordinates
\begin{align}
y_{\beta}\ &= \gamma \mu_\beta x_{\tau(\beta)}  \quad\text{for all $\beta\in \Lambda_1$}, \label{eq:Lambda_1}
\intertext{and}
|y_{\beta'}|\ &\leq \gamma \quad\text{when $\beta'\in  \Lambda\setminus \Lambda_1$.} \label{eq:Lambda_2}
\end{align}
\end{theorem}
\begin{proof}
The implications (c) $\implies $ (a) $\implies$ (b) are plain.
We are verifying (b) $\implies$ (c).
Note that $T$ is injective.  Indeed, if $T(\bx)=0$ for some
nonzero $\bx\in c_0(\Lambda)$, then the fact $T(\bx)$ and $T(\be_\lambda)$ are parallel would imply that
$\bx$ and $\be_\lambda$ are parallel for every $\lambda\in \Lambda$.  But this contradicts to the fact that
$\Lambda$ is infinite and $\bx$ is essentially null.  In particular,   for
any nonzero $\bx=(x_\lambda)_{\lambda\in \Lambda} \in c_0(\Lambda)$,
   its   peak set
  $$
  \Pk(\bx) = \{\lambda\in \Lambda: |x_\lambda|=\|\bx\|_\infty\}.
  $$
is a nonempty proper subset of $\Lambda$.

   For any  $\alpha, \beta \in \Lambda$, we claim that
   $$
   \Pk(T(\be_\alpha))\cap  \Pk(T(\be_\beta)) =\emptyset\quad\text{whenever $\alpha\neq \beta$.}
   $$
In fact, if $\lambda\in \Pk(T(\be_\alpha))\cap  \Pk(T(\be_\beta))$ then both
$T(\be_\alpha)$ and $T(\be_\beta)$ attain their norms at the $\lambda$-coordinate, and thus they are parallel.
This forces $\be_\alpha$ and $\be_\beta$ are parallel, a contradiction.

Consider the disjoint union
$$
\Lambda_1 = \bigcup_{\lambda\in \Lambda} \Pk(T(\be_\lambda)).
$$
We define a surjective map $\tau: \Lambda_1\to \Lambda$ such that
$$
\tau(\beta)= \alpha\quad\text{if and only if}\quad \beta \in \Pk(T(\be_\alpha)).
$$
There is a unimodular scalar  $\mu_\lambda$ such that the norm attaining $\lambda$--coordinate of
$\overline{\mu_\lambda}T(\be_{\tau(\lambda)})$ is positive for every $\lambda$ in $\Lambda_1$.
Replacing $T$ with $Q T$ for a suitable ``diagonal unitary matrix'' $Q$, we can assume that all $\mu_\lambda=1$.

Let $\alpha_1=\tau(\beta_1), \alpha_2=\tau(\beta_2)$ and  $\alpha_3=\tau(\beta_3)$ be three distinct indices for
$\beta_1, \beta_2$ and $\beta_3$ in $\Lambda_1$.
 Consider the linear map $L: \Span(\be_{\alpha_1}, \be_{\alpha_2}, \be_{\alpha_3})\to \Span(\be_{\beta_1}, \be_{\beta_2}, \be_{\beta_3})$
 defined by taking only the $\beta_1$--, $\beta_2$-- and $\beta_3$--coordinates of
 $T(\bx)$ when $\bx\in \Span(\be_{\alpha_1}, \be_{\alpha_2}, \be_{\alpha_3})$.  We can identify $L$ with a $3\times 3$ matrix $A$
 satisfying
 the assumption in Lemma \ref{A-matrix}, from which  we have a positive $\gamma$ such that
 $L(\be_{\alpha_j})=\gamma \be_{\beta_j}$ for $j=1,2,3$.

The above argument shows that
  for any    $\alpha$ in $\Lambda$,   the $\beta$-coordinate of $T(\be_{\alpha})$ is a fixed nonzero scalar $\gamma$
  whenever $\tau(\beta)=\alpha$,
  and all the other $\beta'$-coordinates with $\beta'\in \Lambda_1$ are zero.  In other words,
\begin{align}\label{eq:Tbe}
 T(\be_{\alpha})=\gamma \sum_{\tau(\beta)=\alpha} \be_{\beta} + \bt_\alpha \quad\text{for every $\alpha\in \Lambda$,}
\end{align}
where $\bt_\alpha\in c_0(\Lambda\setminus \Lambda_1)$,
that is, all $\lambda$-coordinate of $\bt_\alpha$ with $\lambda\in \Lambda_1$ are zero.
Since the peak set $\Pk(T(\be_\alpha))\subseteq \Lambda_1$, we have $\|\bt_\alpha\|_\infty < \gamma$.
Note that the above sum must be finite, as $T(\be_\alpha)$ is essentially null.
Replacing $T$ by $T/\gamma$, we can assume that $\gamma=1$, and thus all $\|T(\be_{\alpha})\|_\infty=1$.

We claim that $\|T(\bx)\|_\infty \leq 1$ whenever  $\bx=(x_\lambda)_{\lambda\in \Lambda}$   in $c_0(\Lambda)$ has  norm one.
To see this,
we first assume that $x_{\alpha_1}=1$ and $x_{\alpha_2}=x_{\alpha_3}=0$
with $\alpha_1=\tau(\beta_1)$, $\alpha_2=\tau(\beta_2)$, $\alpha_3=\tau(\beta_3)$
for distinct indices $\alpha_1, \alpha_2, \alpha_3\in \Lambda$ and $\beta_1, \beta_2, \beta_3\in \Lambda_1$.
Since $\bx$ and $\be_{\alpha_1}$   are parallel,   so are $T(\bx)$ and $T(\be_{\alpha_1})=
  \be_{\beta_1}$.  In particular, $\|T(\bx)\|_\infty$ is attained at the $\beta_1$-coordinate of $T(\bx)$.
  On the other hand, $\bx$ is not parallel with $\be_{\alpha_2}, \be_{\alpha_3}$, and thus $T(\bx)$ is not
parallel with $T(\be_{\alpha_2})= \be_{\beta_2}, T(\be_{\alpha_3})=\be_{\beta_3}$.
With $\bx$ playing the role of $\be_{\alpha_1}$, and $T(\bx)$ playing the role of $\be_{\beta_1}$, the above argument
shows that the $\beta_1$-coordinate of $T(\bx)$ is $\|T(\bx)\|_\infty=1$.
In general, for any norm one $\bx=(x_\lambda)_{\lambda\in \Lambda}$ in
$c_0(\Lambda)$,
we may choose distinct indices $\alpha_1=\tau(\beta_1)$, $\alpha_2=\tau(\beta_2)$, $\alpha_3=\tau(\beta_3)$,
and assume its $\alpha_1$-coordinate equals $1$.  Then
\begin{align*}
\Big| \|T(\bx)\|_\infty - \|T(\bx - x_{\alpha_2}\be_{\alpha_2} - x_{\alpha_3}\be_{\alpha_3})\|_\infty\Big| \leq
                  \|T(x_{\alpha_2}\be_{\alpha_2}+ x_{\alpha_3}\be_{\alpha_3})\|_\infty
\end{align*}
implies
\begin{align*}
\Bigg| \|T(\bx)\|_\infty -1\Bigg| &\leq  \left\|x_{\alpha_2} \left(\sum_{\tau(\beta_2)=\alpha_2} \be_{\beta_2} + \bt_{\alpha_2}\right)
               + x_{\alpha_3}\left(\sum_{\tau(\beta_3)=\alpha_3} \be_{\beta_3} + \bt_{\alpha_3} \right)\right\|_\infty \\[5pt]
           &\leq   |x_{\alpha_2}|+ |x_{\alpha_3}|.
\end{align*}
Since $\bx$ is essentially null, we can choose
as small as possible $|x_{\alpha_2}|, |x_{\alpha_3}|$.  It follows that $\|T(\bx)\|_\infty= 1$, as claimed.
In particular, $T$ is an isometry.

Going back to the original bounded linear map $T$, the formula \eqref{eq:Tbe} becomes
$$
T(\be_{\alpha})=\gamma \sum_{\tau(\beta)=\alpha} \mu_\beta\be_{\beta} + \bt_\alpha,
$$
where $\bt_\alpha\in c_0(\Lambda\setminus \Lambda_1)$  with   $\|\bt_\alpha\|_\infty < \gamma$
for every $\alpha\in \Lambda$.
For any $\bx=(x_\alpha)_{\alpha\in \Lambda}= \sum_{\alpha\in \Lambda} x_\alpha\be_\alpha$ in
$c_0(\Lambda)$, the boundedness of $T$ ensures that
$$
T(\bx)= \sum_{\alpha\in \Lambda} x_\alpha T(\be_\alpha)
= \sum_{\alpha\in \Lambda}  x_\alpha\left(\gamma \sum_{\tau(\beta)=\alpha} \mu_\beta\be_{\beta} + \bt_\alpha\right)
= \sum_{\alpha\in \Lambda}\left( \big(\sum_{\tau(\beta)=\alpha} \gamma\mu_\beta x_{\tau(\beta)}\be_{\beta}\big) +x_\alpha\bt_\alpha\right).
$$
It follows \eqref{eq:Lambda_1}.

  To verify \eqref{eq:Lambda_2}, we may assume that  there is $\beta'\in \Lambda\setminus \Lambda_1$
 and a norm one $\bx=(x_\lambda)_{\lambda\in \Lambda}$ in $c_0(\Lambda)$ such that $\by=T(\bx)$ has $\beta'$--coordinate
 with $|y_{\beta'}|>\gamma$.  Since the $\beta$--coordinates of $\by$ are bounded by $\gamma$ for all $\beta\in \Lambda_1$
  due to \eqref{eq:Lambda_1}, we see that
 $\by=T(\bx)$ and $T(\be_\alpha)$  have disjoint peak sets, and thus they
  are not parallel to each other for any $\alpha$ in $\Lambda=\tau(\Lambda_1)$.  However, if $\bx$ attains its norm at the
 ${\alpha_1}$--coordinate, then $\bx$ is parallel with $\be_{\alpha_1}$, and thus $T(\bx)$ is parallel with $T(\be_{\alpha_1})$, a contradiction.
 We thus establish \eqref{eq:Lambda_2}.
 It is now clear that $T/\gamma$ is an into isometry.

 Finally, we note that a (scalar multiple of an) into isometry of $c_0(\Lambda)$ assumes the stated forms \eqref{eq:Lambda_1} and \eqref{eq:Lambda_2},
  due to Holsztynski's Theorem (see, e.g., \cite{JW96}).
\end{proof}

The case for linear parallel/TEA pair preservers of
 $\ell_\infty(\Lambda)$ seems to be more complicated, as an element $\bx=(x_\lambda)_{\lambda\in \Lambda}$
in $\ell_\infty(\Lambda)$ might have empty peak set, that is, all its coordinates $|x_\lambda|< \|x\|_\infty$.
This makes the argument in the proof of Theorem \ref{thm:c_0} cannot be transported directly.
However, $\ell_\infty(\Lambda)\cong C(\beta\Lambda)$, and we can apply the following result for
abelian   C*-algebras.

\begin{theorem}[\cite{LTWW-poa}]\label{thm:C(X)}
Let $X,Y$ be compact Hausdorff spaces containing of at least three points.
Let $T: C(X)\to C(Y)$ be a  bijective  linear map
such that both $T, T^{-1}$ preserve parallel pairs.
Then there is a homeomorphism $\sigma: Y\to X$ and a constant unimodulus function $h\in C(Y)$ such that
$$
Tf(y) = h(y)f(\sigma(y)) \quad\text{for any $f\in C(X)$ and $y\in Y$.}
$$
\end{theorem}

Below is an infinite dimensional analog of Corollary \ref{p-3.1}.

\begin{corollary} \label{cor:inf-linfty}
Let $\Lambda$ be an infinite index set.
Let  $T: \ell_\infty(\Lambda)\rightarrow \ell_\infty(\Lambda)$ be a bijective linear map.
The following conditions are
equivalent to each other.
 \begin{enumerate}[{\rm (a)}]
\item   Both $T$ and $T^{-1}$  send   TEA pairs to TEA pairs.
\item   Both $T$ and $T^{-1}$  send  parallel pairs to parallel pairs.
\item  $T$ is a scalar multiple of a surjective linear isometry.
\end{enumerate}
In this case, there is $\gamma > 0$, a family $\{\mu_\lambda: \lambda\in \Lambda\}$ of unimodular scalars,
 and a  bijective map $\tau: \Lambda\to \Lambda$ such that
for any $\bx=(x_\lambda)_{\lambda\in \Lambda} \in \ell_\infty(\Lambda)$, the image $\by=T(\bx)=(y_\lambda)_{\lambda\in \Lambda}$
has coordinates
\begin{align}\label{eq:Linfty}
y_{\lambda} = \gamma \mu_\lambda x_{\tau(\lambda)}  \quad\text{for all $\lambda\in \Lambda$}.
\end{align}
In other words, $T$ is a scalar multiple of a ``generalized permutation matrix''.
\end{corollary}
\begin{proof}
While the implications (c) $\implies$ (a) $\implies$ (b) are plain, Theorem \ref{thm:C(X)} establishes (b) $\implies$ (c)
when we identify $\ell_\infty(\Lambda)$ with $C(\beta\Lambda)$.
We note that $\Lambda$ consists of all isolated points of $\beta\Lambda$.  Thus the homeomorphism $\sigma$
induces a bijective map $\tau$ from $\Lambda$ onto itself to implement \eqref{eq:Linfty}.
\end{proof}

\begin{problem}\label{problems}
  Can one replace the two direction preserver conditions  by
  that the linear map $T$ sends parallel/TEA pairs to parallel/TEA pairs in Theorem \ref{thm:c_0} and Corollary \ref{cor:inf-linfty},
  as in the finite dimensional case?
\end{problem}

\section*{Acknowledgment}

Li is an affiliate member of the Institute for Quantum Computing, University of Waterloo; his research was partially supported by the Simons Foundation Grant 851334.
M.-C. Tsai, Y.-S. Wang and N.-C. Wong are supported by Taiwan NSTC grants 112-2115-M-027-002,
113-2115-M-005-008-MY2  and 112-2115-M-110-006-MY2,
respectively.

\end{document}